\documentclass[preprint,9pt]{elsarticle}

\usepackage{amsmath}
\usepackage{amssymb}
\usepackage[all]{xypic}
\usepackage[all]{xy}

\usepackage{amsthm}

  \newcommand{\Ell}{L_{2}}
 
 \newcommand{\h}{\mathfrak{H}}

 \newcommand{\E}{\mathbf{E}}
 
\newcommand{\U}{\mathcal{U}}

 \newcommand{\N}{\mathbb{N}}

 \newcommand{\BH}{\mathcal{B}(\mathfrak{H})}
 
 \renewcommand{\ker}{\operatorname{ker}}
\newcommand{\Rang}{\operatorname{Rang}}
 \newcommand{\Real}{\mathbb{R}}
 
 \newcommand{\Complex}{\mathbb{C}}

 \newcommand{\Ele}{L_{2}}

\theoremstyle{statement}
\newtheorem{theorem}{Theorem}[section]

\newtheorem{proposition}[theorem]{Proposition}
\newtheorem{definition}[theorem]{Definition}

\newtheorem{remark}[theorem]{Remark}

\newtheorem{example}[theorem]{Example}
\numberwithin{equation}{section}

\journal{xxxxxxxx}

\begin{document}

\begin{frontmatter}

\title{Frames of subspaces in Hilbert spaces with $W$-metrics}

\author[primitivo]{Primitivo Acosta-Hum\'anez }
\address[primitivo]{Department of Mathematics, Universidad del Norte, Colombia.\\ Email: pacostahumanez@uninorte.edu.co}

\author[Esmeral]{Kevin Esmeral }

\address[Esmeral]{ Department of Mathematics, Centro de Investigaciones y Estudios Avanzados (CINVESTAV), M\'exico.\\ Email: kmesmeral@math.cinvestav.mx}
\author[Ferrer]{Osmin Ferrer }

\address[Ferrer]{ Department of mathematics, Universidad Surcolombiana, Colombia.\\ Email: osmin.ferrer@usco.edu.co}

\begin{abstract}
If $\left(\h,\langle\cdot,\cdot\rangle\right)$ is a Hilbert space and on it we consider the sesquilinear form $\langle\,W\cdot,\cdot\rangle$ so-called $W$-metric, where $W^{*}=W\in\BH$, and $\ker\,W=\{0\}$, then the space $\left(\h,\langle\,W\cdot,\cdot\rangle\right)$ is called Hilbert space with $W$-metric or simply $W$-space. In this paper we investigate the dynamic of frames of subspace on these spaces, where the sense of dynamics refers to the behavior of frames of subspace in $\h_{W}$ (the completion of $\left(\h,\langle\,W\cdot,\cdot\rangle\right)$) comparing with $\h$  and vice versa.
This work is based on the study made in \cite{KEFER,GMMM} on frames in Krein spaces. In a similar way, Casazza and Kutyniok obtained some results in the context of Hilbert spaces, see \cite{CG}.
  We take tools of theory of  $C^{*}$-algebra, and properties of $\BH$, to show that every Hilbert space with $W$-metric $\h_{W}$ with $0\in\sigma(W)$ has a decomposition
$$\h_{W}=\bigoplus_{n\in\N\cup\{\infty\}}\h_{\psi_{n}}^{W},$$
where $\h_{\psi_{n}}^{W}\simeq \Ele(\sigma(W),x\,d\mu_{n}(x))$ are Krein spaces, for every $n\in\N\cup\{\infty\}$. Moreover, we investigate  the dynamics of frames of subspace when the self-adjoint operator $W$ is unbounded.
\end{abstract}


\end{frontmatter}

\section*{Introduction}
The frames were introduced in 1952 by Duffin and Schaeffer, see \cite{DS}, and almost thirty years later, in 1986, Daubechies, Grossmann and Meyer
 used frames to find  expansions in series of functions in $ \Ele (\Real) $ similar to the expansion in series using orthonormal basis, see \cite{DG}. They show that the frames are flexible tools, because work with frames allows us to avoid the linear independence and orthogonality between elements.  In recent years the theory of frames on Hilbert space has been increasing due to the development of new applications, see \cite{CO,Olec,Olecjensen,Frames}.

 The frame of subspaces on Hilbert spaces was studied by Casazza,  Kutyniok in \cite{CG}, and G\u{a}vru\c{t}a in \cite{G}. On the other hand,  the Hilbert spaces with $W$-metric was developed by Azizov, and Iokhvidov in \cite{Iokhvidov}. Now we define
 frame of subspaces on Hilbert spaces  with $W$-metric, but in  \cite{KEFER, KEG} is shown that the properties $ 0\in\rho(W)$ (i.e., $0$ is in the resolvent of $W$), $ 0\in\sigma(W)$ (i.e., $0$ is in the spectrum of $W$)  do influence when we want to see the behavior of the frames.  Therefore, in this work we do it in case of frame of subspaces.

The main motivation comes from theory of $C^{*}$-algebra, where the theorem of spectral representation (see \cite{BS}) is proved for Hilbert spaces. We want to write this important theorem in the sense of frame of subspace and after see it on Hilbert spaces  with a $W$-metric.

This paper is structured as follows. In Section \ref{preliminares}, based on the references \cite{Iokhvidov,Azizov,JBognar}, we will consider some basic aspects related to Krein spaces and Hilbert spaces with $W$-metric.
In Section \ref{main} we present our approach to the study of frame of subspaces. In the subsection \ref{framesubhil} we introduce the frame of subspaces in Hilbert spaces and we set an analogous in the subsection \ref{framesofsubspace} to the case of Krein spaces, i.e.,
is defined frame of subspaces in Krein spaces, which is deduced from the definition of frame in Krein spaces given in \cite{KEFER}, showing some results that are a consequence of the definition. Next, in the subsection \ref{frameswmetric} is considered the Gram operator $W$ bounded and, dependent of $0\in\sigma(W)$,  or $0\notin\sigma(W)$. Later, in subsubsections \ref{frameswmetricregular} and \ref{frameswmetricsingular}, are analyzed the frames of subspaces on regular and singular Krein space (respectively). The main result of this paper in this subsection: \emph{every singular Krein space has a decomposition in direct sum of singular Krein subspaces  $\h_{\psi_{n}}^{W}$, which are isomorphic to $\Ele(\sigma(W),x\,d\mu_{n}(x))$ for every $n\in\N\cup\{\infty\}$}. Finally, in the subsubsection \ref{frameswmetricunbounded} studying the behavior of frames of subspace when the Gram
operator is unbounded. The interested reader in this subject can found some open questions and remarks in Section \ref{remark}.
 \newpage
\section{Preliminaries}\label{preliminares}

\begin{definition}[\textit{Krein Spaces}]
Let $\Re$ be a vector space on $\Complex$. Consider \linebreak $[\cdot,\cdot]:\,\Re\times\Re\longrightarrow\Complex$,  a  sesquilinear form.  The vector space $\left(\Re,\,[\cdot,\cdot]\right)$ is a \textit{Krein space} whether $\Re=\Re^{+}\oplus\Re^{-}$ and  $\left(\Re^{+},\,[\cdot,\cdot]\right)$, $\left(\Re^{-},\,-[\cdot,\cdot]\right)$ are Hilbert spaces, where $\Re^{+}$, $\Re^{-}$ are orthogonal with respect $[\cdot,\cdot]$.
\end{definition}
On $\Re$ define the following scalar product
$$(x_{1},x_{2})=[x_{1}^{+},x_{2}^{+}]-[x_{1}^{-},x_{2}^{-}],\; x_{i}^{\pm}\in\Re^{\pm},\; x_{i}=x_{i}^{+}+x_{i}^{-}.$$
This scalar product makes $\left(\Re,\,(\cdot,\cdot)\right)$  a Hilbert space, which is so-called Hilbert space associated to $\Re.$ Hence, we can take the orthogonal projections on $\Re^{+}$ and $\Re^{-}$ denoted $P^{+}$ and $P^{-}$ respectively. The linear bounded operator $J=P^{+}-P^{-}$ is called \textit{Fundamental Symmetry}, and it satisfies the equality $[x,y]=(Jx,y),\;\forall x,\,y\in\Re.$ Equivalently
\begin{equation}\label{jnorma}
[x,y]_{J}=[Jx,y]=(x,y),\quad\text{and denote} \quad\|x\|_{J}=\sqrt{[x,x]_{J}}\,\forall x,\,y\in\Re.
\end{equation}
\begin{definition}
Let $(\Re,[\cdot,\cdot])$ be a Krein space, Consider $x,y\in\Re$,  we say that $x$ is orthogonal to $y$ if $[x,y]_{J}=0$, and is denoted by $x\perp y$. We say that $x$ is $J$-orthogonal to $y$  if $[x,y]=0$, and is denoted by $x[\perp]y$.
\end{definition}
\begin{definition}
A subspace $V$ of $\Re$  such that $V\cap V^{[\perp]}=\{0\}$ and $V+ V^{[\perp]}=\Re$, where $V^{[\perp]}$ is the orthogonal complement of $V$ with respect $[\cdot,\cdot]$ is called  projectively complete.
\end{definition}
\begin{remark}
  From now on, the closed subspaces taken in this paper are  projectively complete. We denoted by $P_{V}$ and $Q_{V}$ the orthogonal and $J$-orthogonal projections on $V$ respectively. i.e., $P_{V}^{*J}=P_{V}=P_{V}^{2}$ and $Q_{V}^{[*]}=Q_{V}=Q_{V}^{2}$ resp. On the other hand,  in \cite{Iokhvidov} shown that for any closed subspace $V$ its $J$-orthogonal complement $V^{[\perp]}$ and its orthogonal complement $V^{\perp}$ are closed subspaces connected by the formulae
  \begin{align}
    V^{[\perp]}=JV^{\perp},\quad   V^{\perp}=JV^{[\perp]}, \quad   (JV)^{[\perp]}=JV^{[\perp]}\label{subspacepkpk}.
  \end{align}
  By \eqref{subspacepkpk} we note that $JV$  is projectively complete if and only if $V$ is projectively complete. In addition, the condition $V\cap V^{[\perp]}=\{0\}$ said us that every $k\in\Re$ has an unique  $J$-orthogonal projection on $V$, see \cite{Iokhvidov}.
\end{remark}
\begin{definition}\cite{Iokhvidov}
  Let $\Re$ be a Krein space. A system vector $\{e_{i}\}_{i\in I}\subset \Re$, where $I$ is an arbitrary set of indices is called a $J$-orthonormalized system if $[e_{i},e_{j}]=\pm\delta_{i,j}$ for all $i,j\in I$, where $\delta_{i,j}$ is the Kronecker delta.
\end{definition}
\begin{example}
   The simplest example of an $J$-orthonormalized system of $\Re$ is  the union of two arbitrary orthonormalized (in the usual sense) systems from the subspaces $\Re^{+}$ and $\Re^{-}$ respectively.
\end{example}
\begin{definition}\cite{Iokhvidov}
  Let $\Re$ be a Krein space. A $J$-orthonormalized basis in $\Re$ are a  $J$-orthonormalized systems which are (Schauder) bases in $\Re$.
\end{definition}

If  $\Re$ is considered with the scalar product and the generated norm \eqref{jnorma}, we have that  $(\Re,[\cdot,\cdot]_{J})$ is a Hilbert space. But, the main purpose is to study linear operators acting on Krein spaces, the topology of this space is important for issues related to the continuity, and closedness operators and spectral theory etc. This topology is generated by $J$-norm given in \eqref{jnorma}.

 Therefore, arise that some definitions of \emph{operator theory in Hilbert spaces} is satisfies. In addition, if  we talk of continuity of operators, we referred respect to topology generated by $J$-norm. For example
$$\mathcal{B}(\Re)=\left\{T:\Re\rightarrow\Re:\;\text{lineal y}\quad \|T\|=\sup_{x\in\Re\setminus\{0\}}\frac{\|Tx\|_{J}}{\|x\|_{J}}<\infty\right\}.$$
The adjoint of an operator $T$ in Krein spaces $\left(T^{[\ast]}\right)$ satisfies $[T(x),y] = [x,T^{[\ast]}(y)]$, but, we must consider that $T$ have an adjoint operator in the Hilbert space  $(\Re, [\cdot,\cdot]_{J})$  denoted $\left(T^{\ast J}\right)$, where  $J$  is the fundamental symmetry in $\Re$, and there isf a relation between $T^{\ast J}$ and $T^{[\ast]}$, which is $T^{[ \ast ] } = J T^{\ast J}J$. Moreover, let $\Re$ and $\Re'$ be Krein spaces  with fundamental symmetries $J_{\Re}$ and $J_{\Re'}$ respectively, if $T\in\mathcal{B}(\Re,\Re')$ then $T^{[*]_{\Re}}=J_{\Re}T^{*J_{\Re}}J_{\Re'}$. An operator $T\in\mathcal{B}(\Re)$ is said to be self-adjoint if $T=T^{[*]}$, and $J-$self-adjoint whether $T=T^{*J}$, moreover, a linear operator $T$ is said to be positive whether $[Tk,k]\geq0$ for every $k\in\Re$. An operator $T$ is said to be uniformly positive if there exists $\alpha>0$ such that
$[Tk,k]\geq\alpha\|k\|_{J}$ for every $k\in\Re$.

Next result is very important for us, its proof is follows the same argument applied at case  Hilbert space. See in \cite{G}.
\begin{proposition}\label{conmutatividadj}
Let $\Re$ and $\widetilde{\Re}$ be Krein spaces with fundamentals symmetries  $J,\widetilde{J}$ resp. Consider $V\subset\Re$ a closed subspace with $J$-orthogonal projection $Q_{V}:\Re\rightarrow V$, and orthogonal projection $P_{V}:\Re\rightarrow V$.
 Let  $\mathcal{U}: \left(\Re,[\cdot,\cdot]_{J}\right)\rightarrow\left(\widetilde{\Re},[\cdot,\cdot]_{\widetilde{J}}\right)$ and $\mathcal{T}: \Re\rightarrow\widetilde{\Re}$   be  unitary operators. Then is satisfied
  \begin{equation}\label{conUpro}
   \U P_{V}\U^{-1}=P_{\mathcal{U}V},\quad\mathcal{T}Q_{V}\mathcal{T}^{-1}=Q_{\mathcal{T}V},
  \end{equation}
   where $P_{\mathcal{U}V}:\widetilde{\Re}\rightarrow \mathcal{U}V$ is the orthogonal projection on $\mathcal{U}V$, and $P_{\mathcal{T}V}:\widetilde{\Re}\rightarrow \mathcal{T}V$ is the $J$-orthogonal projection on $\mathcal{T}V$.
 In particular, if $\Re=\widetilde{\Re}$ and $J=\widetilde{J}$, then
\begin{equation}\label{conJpro}
P_{JV}=JP_{V}J=P_{V}^{[*]},\quad Q_{JV}=JQ_{V}J=Q_{V}^{*J}.
\end{equation}
\end{proposition}
\begin{proposition}\label{proproyeqp}
  Let $V$ be a closed subspace of $\Re$. The following statements is hold.
  \begin{itemize}
\item [$i).$]  If $P_{V}$ is an orthogonal projection on $V$, then $Q_{V}=P_{V}P_{JV}$ is a $J$-orthogonal projection on $V$.
\item [$ii).$]  If $Q_{V}$ is an orthogonal projection on $V$, then $P_{V}=Q_{V}Q_{JV}$ is a $J$-orthogonal projection on $V$.
  \end{itemize}
\end{proposition}
\begin{proof}
  We prove only $i$,  the proof of $ii$ is analogous. Let $P_{V}$ be an orthogonal projection on $V$, consider $Q_{V}=P_{V}P_{JV}$ which is well defined by \eqref{subspacepkpk} and since that $JV\subset V$, we have  $[Q_{V}x,y]=[P_{V}P_{JV}x,y]= [P_{JV}x,P_{JV}y]= [P_{V}Jx,P_{V}Jy]=[Jx,Jy]=[x,y]$, for all $x,y\in V$. i.e., $Q_{V}x=x$,   $\forall x\in V$.
 On the other hand, by \eqref{conJpro} we get  $Q_{V}^{[*]}=P_{JV}^{[*]}P_{V}^{[*]}=P_{V}P_{JV}=Q_{V}$.  Since that $JV\subset V$, we have that $Q_{V}=P_{V}P_{JV}=P_{JV}=P_{JV}^{2}=(P_{V}P_{JV})^{2}=Q_{V}^{2}$.
\end{proof}


\subsection{Hilbert spaces with $W$-metrics}\label{Wmetrics}
\begin{definition}[$W$-metrics]\mbox{}\\
Let $\h$ be a Hilbert space with scalar product  $\langle\cdot,\cdot\rangle$, and induced norm $\|\cdot\|=\sqrt{\langle\cdot,\cdot\rangle}$. Consider the operator $W=W^{*}\in\BH$ with $\ker W=\{0\}$. The sesquilinear form
\begin{equation}\label{wmetrica}
[\cdot,\cdot]=\langle W(\cdot),\cdot\rangle
\end{equation}
defined on $\h$ is an indefinite inner product so called $W-$metric, or, $W-$inner product, this operator $(W)$ is called \textit{Gram operator}.
\end{definition}
\begin{proposition}\label{propencajehilbertkrein}
 A Hilbert space with a $W$-metric can be densely embedded in a Krein space $\h_{W}$ with fundamental symmetry  $J$.
\end{proposition}
\begin{proof}See \cite{Iokhvidov}.
\end{proof}

\begin{remark}[\textit{Consequences given by the  Gram operator}]
Let $W$ be the Gram operator on $\h$.
\begin{itemize}
\item [$i).$] If $0\in\rho(W)$,  then
\begin{equation}\label{equinorm}
\|W^{-1}\|^{-1}\|x\|^{2}\leq\|x\|_{J}^{2}\leq\|W\|\,\|x\|^{2},\quad\forall x\in\h.
\end{equation}
Therefore
\begin{equation}\label{definkreinregular}
\h_{W}=\overline{\left(\h,[\cdot,\cdot]_{J}\right)}^{\|\cdot\|_{J}}=\left(\h,[\cdot,\cdot]_{J}\right).
\end{equation}

\item [$ii).$] If $0\in\sigma(W)$, then
\begin{equation}
\|x\|_{J}\leq\sqrt{\|W\|}\,\|x\|,\quad\forall x\in\h.
\end{equation}
Hence
\begin{equation}\label{definkreinsingular}
\h_{W}:=\overline{\h}^{\|\cdot\|_{J}}.
\end{equation}
\end{itemize}
\end{remark}

  \begin{definition}
    Let $\left(\h,\langle\cdot,\cdot\rangle\right)$  be a Hilbert space. The Krein space $\h_{W}$ is said to be \emph{regular} whether the Gram operator $W$ is such that $0\in\rho(W)$. Otherwise is said to be \textit{singular.}
  \end{definition}

More details of the regular and singular Krein spaces is recommendable to see  \cite{Iokhvidov}.
\begin{remark}
Consider the polar decomposition  of $W$ given by the formula
\begin{equation}\label{descopolar}
W=J|W|,
 \end{equation}
where the linear operator $J:\left(\ker\;|W|\right)^{\perp}=\overline{\Rang\;|W|}=\h\rightarrow\overline{\Rang\;W}=\h$ is a partial isometry. However,  $\ker\;J=\{0\}$, this imply that $J$ is a unitary operator.
\end{remark}
 \begin{proposition} The operators  $|W|$, $W$ commute with $J$, where $J$ is such that  \eqref{descopolar} is true. Also  $J=J^{*}$.
\end{proposition}
\begin{proof}By the properties of spectral measure we have $W\,|W|=|W|\,W$. Hence, $\left(J|W|-|W|J\right)|W|=0$. i.e $J|W|=|W|J.$ On the other hand, note that $J|W|=W=W^{*}=|W|J^{*}$. Therefore $|W|\left(J-J^{*}\right)=0$. i.e., $J=J^{*}=J^{-1}$,  because $\ker |W|=\{0\}$. Since that $J|W|=|W|J$, we have  $JW=J\left(J|W|\right)=J\left(|W|J\right)=J|W|J=WJ$.
\end{proof}

\begin{definition}
The space $\h_{W}$ is a Krein space with $J$-norm generate by the $J$-inner product
\begin{equation}\label{Jpro}
  [x,y]_{J}=[Jx,y]=\langle WJx,y\rangle=\langle |W|x,y\rangle\quad\forall x,y\in\h,
\end{equation}
where $J$ is the symmetry of Hilbert space $\h$ such that $W=J|W|$.
\end{definition}

\section{Main results}\label{main}
In this section we present an approach to study frame of subspaces in the context of Krein and Hilbert spaces.
\subsection{Frames of subspace on  Hilbert spaces}\label{framesubhil}
Next we consider the frames of subspace in Hilbert spaces. See details in  \cite{CG}. From now on we consider  $I$  like a set of indices, and define
  \begin{equation}
    \ell_{+}^{\infty}(I)=\left\{(x_{i})_{i\in I}\in\ell^{\infty}(I): x_{i}\in \Real_{+} ,\quad\forall i\in I\right\}.
  \end{equation}
  \begin{definition}[\textit{Frame of subspace}]\label{definframesubihilbert}
 Let $\h$ be a Hilbert space with norm $\|\cdot\|$. A family $\left\{V_{i}\right\}_{i\in I}$ of closed subspaces of $\h$  is said to be a \emph{frame of subspace} of Hilbert space $\h$ with respect $(x_{i})_{i\in I}\in\ell_{+}^{\infty}(I)$ (denoted by $\left\{x_{i},V_{i}\right\}_{i\in I}$) whether there are constants $A,B>0$ such that
  \begin{equation}\label{definframesubhil}
    A\|y\|^{2}\leq \sum_{i\in I}x_{i}^{2}\left\|P_{V_{i}}y\right\|^{2}\leq B\|y\|^{2},\; \forall \, y\in\h,
  \end{equation}
  being $P_{V_{i}}:\h\rightarrow V_{i}$ orthogonal projections.  The numbers $A$ and $B$ are called \textit{frame bounds}.
\end{definition}
\subsection{Frames of subspace on  Krein spaces}\label{framesofsubspace}
Here we consider the frame of subspaces in Krein spaces. This work is based in the properties given to Hilbert spaces, which can be found in \cite{CG} together with the study made in \cite{KEFER}.
\begin{definition}\label{definframesubi}
Let $\Re$ be a Krein space with fundamental symmetry $J$.  Consider a family of closed subspaces non degenerated $\left\{V_{i}\right\}_{i\in I}$ of $\Re$  with $Q_{V_{i}}:\Re\rightarrow V_{i}$ their respective $J$-orthogonal projections. Fix $(x_{i})_{i\in I}\in\ell_{+}^{\infty}(I)$,  we say that $\left\{x_{i},V_{i}\right\}_{i\in I}$ is a \textit{frame of subspace} of Krein space whether there are constants $A,B>0$ such that
  \begin{equation}\label{definframesub}
    A\|k\|_{J}^{2}\leq \sum_{i\in I}x_{i}^{2}\left\|Q_{V_{i}}k\right\|_{J}^{2}\leq B\|k\|_{J}^{2},\; \forall \, k\in\Re.
  \end{equation}
  \end{definition}
  \begin{remark}
According to the previous definition and in the same way as the case of Hilbert spaces, the constants  $A$ and $B$ are called  \textit{frame bounds}. Whenever $A=B$, the family $\{x_{i},V_{i}\}_{i\in\,I}$
is called a \textit{B-tight frame} of subspaces. In particular, if $A = B = 1$, then the family $\{x_{i},V_{i}\}_{i\in\,I}$ is called \emph{Parseval frame} of subspaces.  The family $\{x_{i},V_{i}\}_{i\in\,I}$ is said to be \textit{orthonormal basis} of subspaces when
\begin{equation}
  \Re=\bigoplus_{i\in\,I}V_{i}.
\end{equation}
   Moreover,  a frame of subspaces $\{x_{i},V_{i}\}_{i\in\,I}$ is said to be \textit{$x$-uniform}, whenever $x := x_{i} = x_{j}$ for all $i, j \in\, I$. In case that we only have the upper bound, the family $\{x_{i},V_{i}\}_{i\in\,I}$ is called \textit{Bessel sequence} of subspaces with Bessel bound $B$.
\end{remark}
\begin{theorem}[\textit{Equivalence of frames of subspaces}]\label{teorequimarsub}
 Let $\Re$ be a Krein space with fundamental symmetry $J$. Consider a family of  closed subspaces  $\left\{V_{i}\right\}_{i\in I}$ of $\Re$ with $Q_{V_{i}}:\Re\rightarrow V_{i}$ and $P_{V_{i}}:\Re\rightarrow V_{i}$ their $J$-orthogonal and orthogonal projections respectively. Fixed $(x_{i})_{i\in I}\in\ell_{+}^{\infty}(I)$, the following statements are equivalent:
  \begin{itemize}
    \item [$i)$.]  $\left\{x_{i},V_{i}\right\}_{i\in I}$ is a frame of subspace of $\Re$ with frame bounds $A,B$;
    \item [$ii)$.] $\left\{x_{i},JV_{i}\right\}_{i\in I}$ is a frame of subspace of $\Re$ with frame bounds $A,B$;
    \item [$iii)$.] $\left\{x_{i},V_{i}\right\}_{i\in I}$ is a frame of subspace of $\left(\Re, [\cdot,\cdot]_{J}\right)$ with frame bounds $A,B$;
    \item [$iv)$.] $\left\{x_{i},JV_{i}\right\}_{i\in I}$is a frame of subspace of $\left(\Re, [\cdot,\cdot]_{J}\right)$ with frame bounds $A,B$.
  \end{itemize}
\end{theorem}
\begin{proof}
The equivalence of $i$ and $ii$ follows from
\begin{align*}
 \left\|Q_{JV_{i}}k\right\|_{J}^{2}&= \left\|JQ_{JV_{i}}k\right\|_{J}^{2}= \left\|Q_{V_{i}}Jk\right\|_{J}^{2}.
 \end{align*}
The same argument is applied to $\{x_{i},JV_{i}\}_{i\in I}$, together with Proposition  \ref{conmutatividadj}, to prove the equivalence of $iii$ and $iv$, i.e.,
  \begin{align*}
  \left\|P_{JV_{i}}k\right\|_{J}^{2}&= \left\|JP_{JV_{i}}k\right\|_{J}^{2}= \left\|P_{V_{i}}Jk\right\|_{J}^{2}.
 \end{align*}
The equivalence $i$ and $iv$ is proved with Proposition \ref{proproyeqp} as follows: Given $P_{JV_{i}}$ and $P_{V_{i}}$ orthogonal projections on $JV_{i}$ and $V_{i}$ respectively, we define $Q_{V_{i}}=P_{V_{i}}P_{JV_{i}}$, which is a $J$-orthogonal projection on $V_{i}$. Thus
$$\|Q_{V_{i}}k\|_{J}^{2}=\|P_{V_{i}}P_{JV_{i}}Jk\|_{J}^{2}=\|P_{JV_{i}}k\|_{J}^{2}.$$
i.e., $iv\rightarrow i$. On the other side, $i\rightarrow iv$ is follow from $$\|P_{JV_{i}}k\|_{J}^{2}=\|P_{V_{i}}P_{JV_{i}}k\|_{J}^{2}=\|Q_{V_{i}}Jk\|_{J}^{2}.$$
\end{proof}

\begin{proposition}\label{propconteoequi}
Fix $\{x_{i}\}_{i\in I}\in \ell_{+}^{\infty}(I)$. Consider a partition   $\{J_{i}\}_{i\in I}$  of $I$ such that
$I=\bigsqcup_{i\in I}J_{i}$ and $\{k_{i,j}\}_{j\in J_{i}}$ is a sequence of frame for the Krein space $\Re$ with frame bounds $A_{i}, B_{i}>0$. Define $V_{i}=span_{j\in J_{i}}\{k_{i,j}\}$  for all $i\in I$ and choose an $J$-orthonormal basis $\{e_{i,j}\}_{j\in J_{i}}$ for each subspace $V_{i}$. Suppose  that $0<A=\inf_{i\in I}A_{i}\leq B=\sup_{i\in I}B_{i}$, then the following statements are equivalent.
\begin{itemize}
  \item [$i).$] $\{x_{i}k_{i,j}\}_{i\in I, j\in J_{i}} $ is a frame for the Krein space $\Re$.
  \item [$ii).$] $\{x_{i}e_{i,j}\}_{i\in I, j\in J_{i}} $ is a frame for the Krein space $\Re$.
  \item [$iii).$] $\{x_{i},V_{i}\}_{i\in I}$ is a frame of subspace for the Krein space $\Re.$
\end{itemize}
\end{proposition}
\begin{proof}
 Since that $\{k_{i,j}\}_{j\in J_{i}}$ is a sequence of frame for the Krein space $\Re$ with frame bounds $A_{i}, B_{i}>0$, then in \cite{KEFER} shown that is equivalent to say that $\{k_{i,j}\}_{j\in J_{i}}$ is a sequence of frame for the Hilbert space $\left(\Re,\|\cdot\|_{J}\right)$ with same frame bounds. Thus, by Theorem \ref{teorequimarsub} the proof is analogous at case Hilbert space, see \cite{CG}.
\end{proof}

\begin{proposition}\label{marenmar}
Let $\Re$ and $\widetilde{\Re}$ be Krein spaces with fundamental symmetries $J$ and $\widetilde{J}$ respectively. Consider $\U:\left(\Re,[\cdot,\cdot]_{J}\right)\rightarrow\left(\widetilde{\Re},[\cdot,\cdot]_{\widetilde{J}}\right)$  an invertible operator. The family $\{x_{i},V_{i}\}_{i\in I}$ is a frame of subspace for the Hilbert space $\left(\Re,[\cdot,\cdot]_{J}\right)$ if and only if  $\{x_{i},\U V_{i}\}_{i\in I}$ is a frame of subspace for the Hilbert space $\left(\widetilde{\Re},[\cdot,\cdot]_{\widetilde{J}}\right).$
\end{proposition}
\begin{proof}
By Proposition \ref{conmutatividadj} we obtain
\begin{equation*}
\sum_{i\in I}x_{i}^{2}\| P_{V_{i}}k\|_{J}^{2}=\sum_{i\in I}x_{i}^{2}\|\U^{-1} P_{\U V_{i}}\U k\|_{J}^{2}=\sum_{i\in I}x_{i}^{2}\|P_{\U V_{i}} \U k\|_{\widetilde{J}}^{2}.
\end{equation*}
Hence, the equivalence is follows immediately.
\end{proof}
\subsection{Frames of subspaces in Hilbert spaces with $W$-metric}\label{frameswmetric}

   In \cite{KEFER} is proved that the behavior of frame in Hilbert spaces with a $W$-metric is depending of properties $0\in \rho(W)$ or $0\in\sigma(W)$. Next we study the frame of subspaces in this spaces.

\subsubsection{Frames of subspaces on  regular Krein  spaces}\label{frameswmetricregular}
\begin{theorem}\label{makrere}
Let $\h_{W}$ and $V_{i}\subset\h$ for all $i\in I$ be a regular Krein space and closed subspaces of $\h$ respectively.  The family $\{x_{i},V_{i}\}_{i\in I}$ is a frame of subspace for the Hilbert space $\left(\h,\langle\cdot,\cdot\rangle\right)$ if and only if  $\{x_{i},V_{i}\}_{i\in I}$ is a frame of subspace for the regular Krein space $\h_{W}.$
\end{theorem}
\begin{proof}
Since that $\h_{W}$ is a regular Krein space, then the inequality \eqref{equinorm} is satisfied. Hence, if $Q_{V_{i}}=P_{V_{i}}P_{JV_{i}}$, then
  \begin{align}\label{desi1}
\|W^{-1}\|^{-1}\sum_{i\in I}x_{i}^{2}\left\|P_{V_{i}}Jk\right\|^{2}&\leq\sum_{i\in I}x_{i}^{2}\left\|JP_{V_{i}}Jk\right\|_{J}^{2}=\sum_{i\in I}x_{i}^{2}\left\|Q_{V_{i}}k\right\|_{J}^{2}\nonumber\\
&\leq \|W\|\sum_{i\in I}x_{i}^{2}\left\|P_{V_{i}}k\right\|^{2}.
  \end{align}
$\left.\Rightarrow\right]$ Suppose that $\{x_{i},V_{i}\}_{i\in I}$ is a frame of subspace for the Hilbert space $\left(\h,\langle\cdot,\cdot\rangle\right)$ with frame bounds $A,B$, then, by inequality \eqref{desi1} obtain $\{x_{i},V_{i}\}_{i\in I}$ is a frame of subspace for the Hilbert space $\left(\h_{J},[\cdot,\cdot]_{J}\right)$ with frame bounds $A'=\|W^{-1}\|^{-1}\,A$ and $B'=\|W\|\, B$. By Theorem \ref{teorequimarsub} we say that $\{x_{i},V_{i}\}_{i\in I}$ is a frame of subspace for the regular Krein space $\h_{W}.$

$\left[\Leftarrow\right.$ The proof is analogous as in the previous case.
\end{proof}
\subsubsection{Frames of subspaces on singular Krein spaces}\label{frameswmetricsingular}
\begin{theorem}
Let $\h_{W}$ and $V_{i}\subset\h$ for all $i\in I$ be a regular Krein space and closed subspaces of $\h$ respectively.  If $\{x_{i},V_{i}\}_{i\in I}$is a frame of subspace for the Hilbert space $\left(\h,\langle\cdot,\cdot\rangle\right)$, then   $\{x_{i},V_{i}\}_{i\in I}$ is not frame of subspace for the singular Krein space $\h_{W}.$
\end{theorem}
\begin{proof}
Since that $\h_{W}$ is a singular Krein space, then $0\in\sigma_{c}(W)$. Thus, given $\varepsilon>0$, the spectral measure  $\E_{\lambda}$, where
$ W=\int_{\sigma(W)}\,\lambda\,d\E_{\lambda}$,
satisfies  $\E_{\lambda}\left((0,\varepsilon]\right)\neq0.$ But $(0,\varepsilon]=\bigcup_{n\in\N}\left[\frac{\varepsilon}{n},\varepsilon\right],$
hence, there is $n_{0}\in\N$ such that $\E_{\lambda}\left(\left[\frac{\varepsilon}{n_{0}},\varepsilon\right]\right)\neq0$. Now, we  assume $f\in \E_{\lambda}\left(\left[\frac{\varepsilon}{n_{0}},\varepsilon\right]\right)\h\cap V_{j}$ for some $j\in I$ such that $\|f\|=1$ and $\|f\|_{J}\leq1$ (since that $\|f\|_{J}\leq \|\sqrt{|W|}\,\|\,\|f\|$). In this way, if $M=\sup_{j\in I}x_{j}$, then
\begin{align*}
  \sum_{i\in I}x_{i}^{2}\left\|Q_{V_{i}}f\right\|_{J}^{2}&=x_{j}^{2}\left\|\sqrt{|W|}f\right\|^{2}\leq A^{-1}x_{j}^{2}\sum_{i\in I}x_{i}^{2}\left\|P_{V_{i}}\sqrt{|W|}f\right\|^{2}\\
  &\leq\,A^{-1}BM^{2}\left\langle\left(\int_{\sigma(W)}|\lambda|\,d\E_{\lambda}\right)\E_{\lambda}\left(\left[\frac{\varepsilon}{n_{0}},\varepsilon\right]\right)f,f\right\rangle\\
  &=A^{-1}BM^{2}\int_{\sigma(W)}|\lambda|\,\chi_{\left[\frac{\varepsilon}{n_{0}},\varepsilon\right]}(\lambda)d\left(\E_{\lambda}\right)_{f,f}\\
  &\leq \varepsilon BA^{-1}M^{2}\langle\E_{\lambda}\left(\sigma(W)\right)f,f\rangle\\
  &=A^{-1}BM^{2}\varepsilon\|f\|^{2}=A^{-1}BM^{2}\varepsilon.
\end{align*}
Hence, for $\varepsilon\rightarrow 0$,
\begin{equation}\label{infmarsube}
\inf_{\|f\|_{J}\leq 1}\left(\sum_{i\in I}x_{i}^{2}\left\|Q_{V_{i}}f\right\|_{J}^{2}\right)=0.
\end{equation}
Now, if $\{x_{i},V_{i}\}_{i\in I}$ is a frame of subspace for the singular Krein space $\h_{W}$ with frame bounds $C,D>0$, then for $f$ given above, by Theorem \ref{teorequimarsub} and by \eqref{infmarsube} we arrive to $C=0$, which is a contradiction. i.e., $\{x_{i},V_{i}\}_{i\in I}$ is not frame of subspace for the singular Krein space $\h_{W}.$
\end{proof}
\begin{remark}
A Hilbert space  $\h$ with a  $W$-metric arbitrary can be embedded densely  in a Krein space $\h_{W}$. For instance, we note that if $\h_{W}$ is regular Krein space, then the frames of subspaces are transferable of $\h$ to $\h_{W}$.
This happens because $\left(\h_{W},[\cdot,\cdot]_{J}\right)=\left(\h,[\cdot,\cdot]_{J}\right)$. i.e.,  we change the norm $\|\cdot\|\rightsquigarrow\|\cdot\|_{J}$ in  $\h$, which are equivalent. But, when the Krein space $\h_{W}$ is singular, the frames of subspace are not transferable, because the property $0\in\sigma(W)$ has strong influence. Thence, we must find a way to extend the frames of subspace for the Hilbert space $\h$ to the singular Krein space $\h_{W}$.
\end{remark}
A way to extend frame of subspace in a Hilbert space $\h$ to a singular Krein space $\h_{W}$ is made as follows.
\begin{theorem}\label{singularframe}
 Let $\h_{W}$  be a singular Krein space. There is a invertible operator $\mathcal{U}:\h\rightarrow\h_{W}$ such that:
 \begin{itemize}
 \item [$i).$] If $\{x_{i},V_{i}\}_{i\in I}$ is  frame of subspace for the Hilbert space $\left(\h,\langle\cdot,\cdot\rangle\right)$, then $\{x_{i},\mathcal{U} V_{i}\}_{i\in I}$ is a frame of subspace for the singular Krein space $\h_{W}.$
\item [$i).$] If $\{x_{i},V_{i}\}_{i\in I}$ is a frame of subspace for the singular Krein space $\h_{W}$, then $\{x_{i},\mathcal{U}^{-1}V_{i}\}_{i\in I}$  is a frame for the Hilbert space $\left(\h,\langle\cdot,\cdot\rangle\right)$.
    \end{itemize}
\end{theorem}
\begin{proof}
In \cite{KEFER} was proved that the operator $\sqrt{|W|}:\h\subset\h_{W}\rightarrow\h$, satisfies $\left\|\sqrt{|W|}k\right\|^{2}=\langle \sqrt{|W|}k,\sqrt{|W|}k\rangle=\langle|W|k,k\rangle=\|k\|_{J}^{2}.$ i.e., $\sqrt{|W|}\in\mathcal{B}(\h,\h_{W})$ is an isometry.  Therefore, this isometry has an unitary extension on  $\h_{W}$,  denoted  $\widehat{\sqrt{|W|}}$. Hence, considering $\mathcal{U}=\widehat{\sqrt{|W|}},$ the implications $i)$ and $ii)$ are satisfied immediately with help of the theorems \ref{teorequimarsub} and \ref{marenmar}.
\end{proof}

The following well known result, see for example \cite{BS}, is useful for our main purpose.
\begin{proposition}\label{propdecomW}(Spectral theorem-multiplication operator form) Let $A$
be a bounded self-adjoint operator on $\h$, a separable Hilbert space. Then,
\begin{equation}\label{decompositionW}
 \h =\bigoplus_{n\in\N\cup\{\infty\}}\h_{\psi_{n}},
\end{equation}
and there exist measures $\{\mu_{n}\}_{n=1}^{N} (N= 1,2,... \text{or}\;\infty)$ on $\sigma(A)$ and an unitary operator
\begin{equation}\label{decompositionW0}
T: \h_{\psi_{n}} \rightarrow \Ell(\sigma(A), d\mu_{n})
\end{equation}
such that $\left(TAT^{-1}\psi\right)_{n}(\lambda)=\lambda\psi_{n},\quad n\in\N\cup\{\infty\}$ where we write an element\linebreak $\psi\in\bigoplus_{n=1}^{N}\Ell(\sigma(A), d\mu_{n})$ as an $N$-tuple
$\left(\psi_{1}(\lambda),\psi_{2}(\lambda),\ldots,\psi_{N}(\lambda)\right)$.
\end{proposition}
In the previous proposition the realization of $A$ is called a \emph{spectral representation}, which lead us to the following result.
\begin{theorem}\label{theoremdeco}
  Let $\h$ be a separable Hilbert space, let $W$ be the Gram operator defined on $\h$ such that $0\in\sigma(W)$. Then the Krein space $\h_{W}$ have an orthonormal basis of subspaces.
\end{theorem}
\begin{proof}
  Note that the Gram operator $W$ is self-adjoint, then by Proposition \ref{propdecomW} we have \eqref{decompositionW}
and there exist measures $\{\mu_{n}\}_{n=1}^{N} (N= 1,2,... \text{or}\;\infty)$ on $\sigma(A)$ such that $\h_{\psi_{n}}\simeq\Ell(\sigma(W),d\mu_{n})$.   Hence $\left\{\{1\},\h_{\psi_{n}}\right\}_{n\in\N\cup\{\infty\}}$ is a frame of subspace for the Hilbert space $\h$ with frame bounds 1. Indeed, $\left\{\h_{\psi_{n}}\right\}_{n\in\N\cup\{\infty\}}$ by  \eqref{decompositionW} is an orthonormal basis of subspace of $\h$. By Theorem \ref{singularframe} we say that $\left\{\{1\},\U\h_{n}\right\}_{n\in\N\cup\{\infty\}}$ is a Parseval frame of subspace for the singular Krein space $\h_{W}$. Indeed,  $\left\{\{1\},\U\h_{n}\right\}_{n\in\N\cup\{\infty\}}$ is an orthonormal basis of subspace of $\h_{W}$. Thus, we conclude that
\begin{equation}\label{decompositionW2}
  \h_{W}=\bigoplus_{n\in\N\cup\{\infty\}}\U\h_{\psi_{n}}.\qedhere
\end{equation}
\end{proof}
\begin{remark}
 Owing to $0\notin\sigma(W)$ and  by \eqref{definkreinregular} we see that
  \begin{equation}
    \h_{W}=\left(\h,[\cdot,\cdot]_{J}\right)=\bigoplus_{n\in\N\cup\{\infty\}}\left(\h_{\psi_{n}},[\cdot,\cdot]_{J}\right).
  \end{equation}
\end{remark}
\begin{theorem}
 Let $\h_{W}$ be a singular Krein space with Gram operator $W$. Then $\h_{W}$ has a decomposition as follows
\begin{equation}\label{decompositionW3}
  \h_{W}=\bigoplus_{n\in\N\cup\{\infty\}}\h_{\psi_{n}}^{W}.
\end{equation}
Furthermore, there exist measures $\{\mu_{n}\}_{n=1}^{N} (N= 1,2,... \text{or}\;\infty)$ on $\sigma(W)$ such that $\h_{n}^{W}\simeq\Ell(\sigma(W),xd\mu_{n}(x))$ are  Krein spaces for every $n\in\N\cup\{\infty\}$.
\end{theorem}
\begin{proof}
Since that the separable Krein space $\h_{W}$ is singular, the Gram operator $W$ is such that $0\in\sigma(W)$. Hence, by Theorem \ref{theoremdeco} the equality \eqref{decompositionW2} is satisfies. Define
 \begin{equation}
 \h_{\psi_{n}}^{W}:=\U\h_{\psi_{n}},\quad\forall n\in\N\cup\{\infty\}.
 \end{equation}
We claim to show  $\h_{\psi_{n}}^{W}\simeq\Ele(\sigma(W),x\,d\mu_{n}(x))$. Thence,  fixed $n\in\N\cup\{\infty\}$, let $\Ele\left(\sigma(W),d\mu_{n}\right)$ be a Hilbert space,  where $\mu_{n}$ is a Lebesgue measure. On it Hilbert space we define the bounded and self-adjoint operator given by \linebreak$\left(W_{x}f\right)(x)=xf(x)$. It linear operator is such that $\ker W_{x}=\{0\}$ due to  $\mu_{n}\left(\mathrm{Id}_{\sigma(W)}^{-1}\{0\}\right)=0$, where $\mathrm{Id}_{\sigma(W)}x=x$. In effect, if  $\left(W_{x}f\right)(x)=0$,  for all $x\in\sigma(W)$, then given $\varepsilon\in\N$,  we take the measurable's sets $$M_{\varepsilon}=\left\{x\in\sigma(W):|f(x)|>\varepsilon\right\},$$ and we obtain
\begin{align*}
 0&=\|W_{x}f\|^{2}=\int_{\sigma(W)\setminus\mathrm{Id}_{\sigma(W)}^{-1}\left(\{0\}\right)}|f(x)|^{2}|x|^{2}d\mu_{n},\quad\text{}\\
 &=\int_{M_{\varepsilon}}|f(x)|^{2}|x|^{2}d\mu_{n}(x)\geq\varepsilon^{2}\int_{M_{\varepsilon}\setminus\mathrm{Id}_{\sigma(W)}^{-1}\left(\{0\}\right)}|x|^{2}d\mu_{n}(x).
 \end{align*}
i.e $\mu_{n}\left(M_{\varepsilon}\right)=0,\quad\forall \varepsilon\in\N$. Consequently, if $$M^{+}=\left\{x\in\sigma(W):|f(x)|>0\right\}=\bigcup_{m\in\N}M_{\frac{1}{m}},$$ then $\mu_{n}\left(M^{+}\right)=0.$
Thus, we conclude $|f|=0$ almost everywhere in $\sigma(W)$. i.e., $f=0$ almost everywhere in $\sigma(W)$.

Now, if on the Hilbert space $\Ele(\sigma(W),d\mu_{n}(x))$  we take the Gram operator $W_{x}$, then
$$\overline{\Ele(\sigma(W), d\mu_{n}(x))}^{\|\cdot\|_{J}}\simeq \Ele(\sigma(W),x\, d\mu_{n}(x))$$
is a singular Krein space, where
$$\Ele(\sigma(W),x d\mu_{n}(x)):=\left\{f\in\Ele(\sigma(W),d\mu_{n})\,:\int_{\sigma(W)}|f(x)|^{2}|x|\;d\mu_{n}(x)<\infty\right\}$$
and $\|f\|_{J}^{2}=[f,f]_{J}=\langle |W_{x}|f,f\rangle=\int_{\sigma(W)}|f(x)|^{2}\left|x\right|\,d\mu_{n}(x)$.

On the other hand, Let $F:\Ele\left(\sigma(W),d\mu_{n}\right)\longrightarrow \Ele\left(\sigma(W),\, x\,d\mu_{n}(x)\right)$ be a linear operator given by $\displaystyle(Fg)(x)=\frac{g(x)}{\psi(x)}$, where the function $\psi(x)$ is measurable and $|\psi(x)|^{2}=|x|$ almost everywhere in $\sigma(W)$.  For $n\in\N\cup\{\infty\}$ fixed is satisfies that $\mu_{n}\left(|\psi|^{-1}\left\{0\right\}\right)=\mu_{n}\left(|\mathrm{Id}_{\sigma(W)}|^{-1}\left\{0\right\}\right)=0$, consequently
$$\|Ff\|_{J}^{2}=\int_{\sigma(W)}|f(x)|^{2}|\psi^{-1}(x)|^{2}|x|\,d\mu_{n}(x)=\int_{\sigma(W)}|f(x)|^{2}\,d\mu_{n}(x)=\|f\|^{2}.$$
In addition, the linear operator $F$ has inverse  which is  well defined and is given by $\left(F^{-1}f\right)(x)=\psi(x)f(x).$ i.e., $F$ is an invertible operator.

In conclusion,  the theorem is proved from the diagram
\begin{equation}
  \xymatrix{
  \ar[d]_{\U} \h_{\psi_{n}}\ar[r]^{\hspace{-1cm}G}& \ar[d]^{F}\Ele(\sigma(W),d\mu_{n}(x))\\
\h_{\psi_{n}}^{W}\ar[r]&\Ele\left(\sigma(W),x\,d\mu_{n}(x)\right),
}
\end{equation}
where $G$ is an invertible operator defined in $\h_{\psi_{n}}$ on $\Ele(\sigma(W),d\mu_{n}(x))$.\qedhere
\end{proof}

 \subsubsection{The case:  Gram operator $W$ is unbounded. }\label{frameswmetricunbounded}

Unfortunately,  we cannot obtain  similar results when the Gram operator $W$ on Hilbert space $\h$ is unbounded. Mainly because the tools used are based on specific properties and results given for the $C^{*}$-algebra $\BH$. In the paper \cite{KEFER}  was studied  the behavior of frames in this case. Next we use such results to show the behavior of frames of subspaces of Hilbert spaces with $W$-metric where Gram operator $W$ is unbounded. For more details see \cite{KEFER}.

The Gram operator $W$ is well defined with dense domain  $\mathcal{D}_{W}\subsetneq\h$. Hence, the  $W-$metric $[\cdot,\cdot]=\langle W\cdot,\cdot\rangle$  just is defined to $x,y\in\mathcal{D}_{W}=\mathcal{D}_{W^{*}}$. The polar decomposition ($W=J|W|$) allows us to define
$$[x,y]_{J}:=\langle |W|x,y\rangle,\quad\forall x,\,y\in\mathcal{D}_{W},$$
 and by Proposition \ref{propencajehilbertkrein} we get
\begin{equation}
\h_{W}:=\overline{\mathcal{D}_{W}}^{\|\cdot\|_{J}}.
\end{equation}
\begin{proposition}\label{propconteoequinoaco}
Let $\{x_{i}\}_{i\in I}\in \ell_{+}^{\infty}(I)$. Consider a partition   $\{J_{i}\}_{i\in I}$  of $I$ such that
$I=\bigsqcup_{i\in I}J_{i}$ and $\{k_{i,j}\}_{j\in J_{i}}$ a sequence of frame for the Hilbert space $\left(\h.\langle\cdot,\cdot\rangle\right)$ with frame bounds $A_{i}, B_{i}>0$. We assume $V_{i}=span_{j\in J_{i}}\{k_{i,j}\}$  and $0<A=\inf_{i\in I}A_{i}\leq B=\sup_{i\in I}B_{i}$ for all $i\in I$ . Hence,   if $\{x_{i},V_{i}\}_{i\in I}$ is a frame of subspace for the Hilbert space $\left(\h.\langle\cdot,\cdot\rangle\right)$, then $\{x_{i},V_{i}\}_{i\in I}$ is  not a frame of subspace for the Krein space $\h_{W}.$
\end{proposition}
\begin{proof}
   We have that $\{x_{i}k_{i,j}\}_{i\in I, j\in J_{i}} $  is a frame for the Hilbert space $\h$ (case Hilbert space see \cite{CG}), but in \cite{KEFER} was proved that $\{x_{i}k_{i,j}\}_{j\in J_{i}}$ is not a frame for the Krein space $\h_{W}$ when the Gram operator $W$ is unbounded. Thus, by Proposition \ref{propconteoequi}, the family $\{x_{i},V_{i}\}_{i\in I}$ is  not a frame of subspace for the Krein space $\h_{W}.$
\end{proof}
\begin{theorem}
Let $\h_{W}$  be a  Krein space where the Gram operator $W$ is unbounded and  $0\notin\sigma(W)$. Let $G:\,\mathcal{D}_{G}\subset\h_{W}\rightarrow\h$ be given by $G=\sqrt{|W|}$, the following statements hold.
\begin{itemize}
\item [$i).$] If $\{x_{i},V_{i}\}_{i\in I}$ is  frame of subspace for the Hilbert space $\left(\h,\langle\cdot,\cdot\rangle\right)$, then $\{x_{i},G^{-1} V_{i}\}_{i\in I}$ is a frame of subspace for the singular Krein space $\h_{W}.$
\item [$ii).$] If $\{x_{i},V_{i}\}_{i\in I}$ is a frame of subspace for the singular Krein space $\h_{W}$, then $\{x_{i},G\,V_{i}\}_{i\in I}$  is a frame for the Hilbert space $\left(\h,\langle\cdot,\cdot\rangle\right)$.
    \end{itemize}
 \end{theorem}
 \begin{proof}
Since that $0\in\rho(W)$, the linear operator $G$ is invertible ( See \cite{KEFER}). Thus, by Proposition \ref{marenmar}, the proof is hold.
 \end{proof}

When the Gram operator $W$ is unbounded with $0\in\sigma(W)$ we get an analogous result as Theorem \ref{singularframe}.
\begin{theorem}
 Let $\h_{W}$  be a  Krein space where the Gram operator $W$ is unbounded with $0\in\sigma(W)$. There is an invertible operator $\mathcal{U}:\h\rightarrow\h_{W}$ such that:
 \begin{itemize}
 \item [$i).$] If $\{x_{i},V_{i}\}_{i\in I}$ is a frame of subspace for the Hilbert space $\left(\h,\langle\cdot,\cdot\rangle\right)$, then $\{x_{i},\mathcal{U} V_{i}\}_{i\in I}$ is a frame of subspace for the singular Krein space $\h_{W}.$
\item [$ii).$] If $\{x_{i},V_{i}\}_{i\in I}$ is a frame of subspace for the singular Krein space $\h_{W}$, then $\{x_{i},\mathcal{U}^{-1}V_{i}\}_{i\in I}$  is a frame for the Hilbert space $\left(\h,\langle\cdot,\cdot\rangle\right)$.
    \end{itemize}
\end{theorem}
\begin{proof}
 In \cite{KEFER} was proved that for $0\in\sigma(W)$,  the linear operator $G:\,\mathcal{D}_{G}\subset\h_{W}\rightarrow\h$, which is given by $G=\sqrt{|W|}$,  has an unique unitary extension $U:=\widehat{G}:\,\h_{W}\rightarrow\h.$ Thence by Proposition \ref{marenmar} the statement is hold.
\end{proof}
\section{Final remarks}\label{remark}
In this paper we study Hilbert spaces with bounded Gram operator. When the Gram operator $W$ is unbounded, our main  result is not satisfied whether $\h$ cannot be decomposed. In such case, we can only state that if $\h$ satisfies \eqref{decompositionW} in some way, then the decomposition \eqref{decompositionW3} is  obtained  to $\h_{W}$.\\

The following open questions arose during the writing of this paper.
\begin{enumerate}
  \item [$i).$]  Is it possible to solve  differential equations with frames on some Hilbert or Krein spaces?
  \item [$ii).$] In \cite{AH,AH2,AHMRW} were studied relations between differential Galois theory and the solution of the Schr\"odinger equation over separable Hilbert space ($L_{2}$),  is it possible to obtain similar results in the context of Krein and Hilbert frames subspaces?.
  \item [$iii).$] Is it possible to study  partial differential equations with frames on some Hilbert or Krein spaces?
\item [$iv).$]  How can we write a quantum mechanics formalism in the context on frame of subspaces in Hilbert or Krein  spaces?
\item [$v.)$] How can we use Banach algebras instead of Hilbert spaces to study frame theory?
\item [$vi).$] How can we relate the theory of frame with Weyl $C^{*}$-algebra?
\item [$vii).$] What happens in the case of tensor product on vector space in frame theory?
\end{enumerate}
\section*{Acknowledgements}
Thank  P. G\u{a}vru\c{t}a for his observations and recommendations. The first author is partially supported by the MICIIN/FEDER grant number $MTM2009–06973$, by the Generalitat de Catalunya grant number $2009SGR859$ and by DIDI – Universidad del Norte. The second author is supported by CONACYT and CINVESTAV. The third author is supported by Universidad Surcolombiana.

\bibliographystyle{model1c-num-names}

\end{document}